\documentclass[11pt]{amsart}

\usepackage{anysize} \marginsize{1.3in}{1.3in}{1in}{1in}
\usepackage{comment}
\usepackage{xcolor}
\usepackage{amsmath}
\usepackage{mathtools}
\usepackage{booktabs}
\usepackage[all]{xy}
\usepackage[utf8]{inputenc}
\usepackage{varioref}
\usepackage{amsfonts}
\usepackage{amssymb}
\usepackage{mathabx}
\usepackage{bbm}
\usepackage{esint}
\usepackage{graphicx}
\usepackage{tikz}
\usepackage{empheq}
\usepackage{enumitem}
\usepackage{tikz-cd}
\usepackage{todonotes}
\usetikzlibrary{matrix,arrows,decorations.pathmorphing}
\usepackage{mathrsfs}
\usepackage[hypertexnames=false,backref=page,pdftex,
 	pdfpagemode=UseNone,
 	breaklinks=true,
 	extension=pdf,
 	colorlinks=true,
 	linkcolor=blue,
 	citecolor=red,
 	urlcolor=blue,
 ]{hyperref}

\newcommand{\C}{{\mathbb C}}

\newcommand{\Q}{{\mathbb Q}}
\newcommand{\R}{{\mathbb R}}

\newcommand{\Z}{{\mathbb Z}}

\newcommand{\dd}{{{\mathrm{d}}}}

\renewcommand{\to}[1][]{\xrightarrow{\ #1\ }}

\newcommand{\cp}{\mathbb C \mathrm P}
\newcommand{\rp}{\mathbb R \mathrm P}

\makeatletter
\newcommand*{\da@rightarrow}{\mathchar"0\hexnumber@\symAMSa 4B }
\newcommand*{\da@leftarrow}{\mathchar"0\hexnumber@\symAMSa 4C }
\newcommand*{\xdashrightarrow}[2][]{%
  \mathrel{%
    \mathpalette{\da@xarrow{#1}{#2}{}\da@rightarrow{\,}{}}{}%
  }%
}
\newcommand{\xdashleftarrow}[2][]{%
  \mathrel{%
    \mathpalette{\da@xarrow{#1}{#2}\da@leftarrow{}{}{\,}}{}%
  }%
}
\newcommand*{\da@xarrow}[7]{%
  \sbox0{$\ifx#7\scriptstyle\scriptscriptstyle\else\scriptstyle\fi#5#1#6\m@th$}%
  \sbox2{$\ifx#7\scriptstyle\scriptscriptstyle\else\scriptstyle\fi#5#2#6\m@th$}%
  \sbox4{$#7\dabar@\m@th$}%
  \dimen@=\wd0 %
  \ifdim\wd2 >\dimen@
    \dimen@=\wd2 %
  \fi
  \count@=2 %
  \def\da@bars{\dabar@\dabar@}%
  \@whiledim\count@\wd4<\dimen@\do{%
    \advance\count@\@ne
    \expandafter\def\expandafter\da@bars\expandafter{%
      \da@bars
      \dabar@ 
    }%
  }%
  \mathrel{#3}%
  \mathrel{%
    \mathop{\da@bars}\limits
    \ifx\\#1\\%
    \else
      _{\copy0}%
    \fi
    \ifx\\#2\\%
    \else
      ^{\copy2}%
    \fi
  }%
  \mathrel{#4}%
}
\makeatother

\makeatletter
\newsavebox\myboxA
\newsavebox\myboxB
\newlength\mylenA

\newcommand*\xtilde[2][0.8]{%
    \sbox{\myboxA}{$\m@th#2$}%
    \setbox\myboxB\null
    \ht\myboxB=\ht\myboxA%
    \dp\myboxB=\dp\myboxA%
    \wd\myboxB=#1\wd\myboxA
    \sbox\myboxB{$\m@th\widetilde{\copy\myboxB}$}
    \setlength\mylenA{\the\wd\myboxA}
    \addtolength\mylenA{-\the\wd\myboxB}%
    \ifdim\wd\myboxB<\wd\myboxA%
       \rlap{\hskip 0.5\mylenA\usebox\myboxB}{\usebox\myboxA}%
    \else
        \hskip -0.5\mylenA\rlap{\usebox\myboxA}{\hskip 0.5\mylenA\usebox\myboxB}%
    \fi}

\newbox\usefulbox

\def\getslant #1{\strip@pt\fontdimen1 #1}

\def\xxtilde #1{\mathchoice
 {{\setbox\usefulbox=\hbox{$\m@th\displaystyle #1$}%
    \dimen@ \getslant\the\textfont\symletters \ht\usefulbox
    \divide\dimen@ \tw@ 
    \kern\dimen@ 
    \xtilde{\kern-\dimen@ \box\usefulbox\kern\dimen@ }\kern-\dimen@ }}
 {{\setbox\usefulbox=\hbox{$\m@th\textstyle #1$}%
    \dimen@ \getslant\the\textfont\symletters \ht\usefulbox
    \divide\dimen@ \tw@ 
    \kern\dimen@ 
    \xtilde{\kern-\dimen@ \box\usefulbox\kern\dimen@ }\kern-\dimen@ }}
 {{\setbox\usefulbox=\hbox{$\m@th\scriptstyle #1$}%
    \dimen@ \getslant\the\scriptfont\symletters \ht\usefulbox
    \divide\dimen@ \tw@ 
    \kern\dimen@ 
    \xtilde{\kern-\dimen@ \box\usefulbox\kern\dimen@ }\kern-\dimen@ }}
 {{\setbox\usefulbox=\hbox{$\m@th\scriptscriptstyle #1$}%
    \dimen@ \getslant\the\scriptscriptfont\symletters \ht\usefulbox
    \divide\dimen@ \tw@ 
    \kern\dimen@ 
    \xtilde{\kern-\dimen@ \box\usefulbox\kern\dimen@ }\kern-\dimen@ }}%
 {}}

\newcommand*\xoverline[2][0.75]{%
    \sbox{\myboxA}{$\m@th#2$}%
    \setbox\myboxB\null
    \ht\myboxB=\ht\myboxA%
    \dp\myboxB=\dp\myboxA%
    \wd\myboxB=#1\wd\myboxA
    \sbox\myboxB{$\m@th\overline{\copy\myboxB}$}
    \setlength\mylenA{\the\wd\myboxA}
    \addtolength\mylenA{-\the\wd\myboxB}%
    \ifdim\wd\myboxB<\wd\myboxA%
       \rlap{\hskip 0.5\mylenA\usebox\myboxB}{\usebox\myboxA}%
    \else
        \hskip -0.5\mylenA\rlap{\usebox\myboxA}{\hskip 0.5\mylenA\usebox\myboxB}%
    \fi}

\def\xxoverline #1{\mathchoice
 {{\setbox\usefulbox=\hbox{$\m@th\displaystyle #1$}%
    \dimen@ \getslant\the\textfont\symletters \ht\usefulbox
    \divide\dimen@ \tw@ 
    \kern\dimen@ 
    \overline{\kern-\dimen@ \box\usefulbox\kern\dimen@ }\kern-\dimen@ }}
 {{\setbox\usefulbox=\hbox{$\m@th\textstyle #1$}%
    \dimen@ \getslant\the\textfont\symletters \ht\usefulbox
    \divide\dimen@ \tw@ 
    \kern\dimen@ 
    \xoverline{\kern-\dimen@ \box\usefulbox\kern\dimen@ }\kern-\dimen@ }}
 {{\setbox\usefulbox=\hbox{$\m@th\scriptstyle #1$}%
    \dimen@ \getslant\the\scriptfont\symletters \ht\usefulbox
    \divide\dimen@ \tw@ 
    \kern\dimen@ 
    \xoverline{\kern-\dimen@ \box\usefulbox\kern\dimen@ }\kern-\dimen@ }}
 {{\setbox\usefulbox=\hbox{$\m@th\scriptscriptstyle #1$}%
    \dimen@ \getslant\the\scriptscriptfont\symletters \ht\usefulbox
    \divide\dimen@ \tw@ 
    \kern\dimen@ 
    \xoverline{\kern-\dimen@ \box\usefulbox\kern\dimen@ }\kern-\dimen@ }}%
 {}}
\makeatother

\makeatletter
\newcommand{\mylabel}[2]{#2\def\@currentlabel{#2}\label{#1}}
\makeatother

\makeatletter
\newcommand{\Mac}{}
\DeclareRobustCommand{\Mac}{%
  M%
  \raisebox{\dimexpr\fontcharht\font`M-\height}{%
    \check@mathfonts\fontsize{\sf@size}{0}\selectfont
    c%
  }%
}
\makeatother

\newtheoremstyle{citing}
  {}
  {}
  {\itshape}
  {}
  {\bfseries}
  {\textbf{.}}
  {.5em}
  {\thmnote{#3}}

\theoremstyle{plain}
\newtheorem{theorem}{Theorem}

\newtheorem{lemma}[theorem]{Lemma}
\newtheorem{corollary}[theorem]{Corollary}

\newtheorem{proposition}[theorem]{Proposition}

\theoremstyle{remark}
\newtheorem{example}[theorem]{Example}

\theoremstyle{definition}

\numberwithin{equation}{section}

\theoremstyle{remark}

{\theoremstyle{citing}
}

{\theoremstyle{definition}
}

\title[Maximal Weinstein neighborhoods of symmetric R-spaces and their symplectic capacities]{Maximal Weinstein neighborhoods of symmetric R-spaces and their symplectic capacities}

\author{Johanna Bimmermann}
\address{Fakultät für Mathematik, Ruhr-Universität Bochum, Universitätsstrasse 150, 44801 Bochum, Germany}
\email{johanna.bimmermann@rub.de}

\let\origmaketitle\maketitle
\def\maketitle{
  \begingroup
  \def\uppercasenonmath##1{} 
  \let\MakeUppercase\relax 
  \origmaketitle
  \endgroup
}

\begin{document}
\thispagestyle{empty}

\begin{abstract}
Symmetric \( R \)-spaces can be characterized as real forms of Hermitian symmetric spaces, and as such, they are all embedded as Lagrangian submanifolds. We show that their maximal Weinstein tubular neighborhoods are dense and use this property to compute both the Gromov width and the Hofer--Zehnder capacity of the corresponding disc (co)tangent bundles of the symmetric \( R \)-spaces.
\end{abstract}

\maketitle

\setlength{\parindent}{1em}
\setcounter{tocdepth}{1}



\section{Introduction}\label{symR}

Symmetric \( R \)-spaces \( N \) are a special class of real flag manifolds that also possess the structure of compact-type symmetric spaces \( (N, g) \). More precisely, fix a semisimple (non-compact) Lie group \( G \) and a parabolic subgroup \( P \). Then, the (compact) coset space \( N = G / P \) is called a real flag manifold. If the action of the maximal compact subgroup \( K \subset G \) is transitive, we obtain a \( K \)-invariant metric on \( N \). If the pair \( (K, H) \), where \( H = K \cap P \), forms a symmetric pair, then \( N \) is called a symmetric \( R \)-space. These spaces have been classified (see Appendix~\ref{classification}), and the list includes many notable examples: 
$$
\mathrm{Gr}_{\mathbb{R}} (p,q), \mathrm{Gr}_{\mathbb{C}} (p,q), \mathrm{Gr}_{\mathbb{H}}(p,q), \mathrm{SO}(n), \mathrm{U}(n), \mathrm{Sp}(n), Q_{p,q}(\mathbb{R}), Q_n(\mathbb{C}), \mathbb{O}\mathrm{P}^2, \ldots
$$
\subsection*{Complexification}

The dual description of \( N \) as a homogeneous space,
\[
N \cong K / H \cong G / P,
\]
gives rise to two natural complexifications: \( K^{\mathbb{C}} / H^{\mathbb{C}} \) and \( G^{\mathbb{C}} / P^{\mathbb{C}} \). These two complexifications, however, are not the same. The space \( K^{\mathbb{C}} / H^{\mathbb{C}} \) is \( K \)-equivariantly biholomorphic to the tangent bundle \( TN \), equipped with an adapted complex structure (see~\cite[Thm.~2.1]{Tum23}\footnote{The identification in~\cite[Thm.~2.1]{Tum23} is a diffeomorphism; we include a sketch of the proof in Appendix~\ref{biholo} explaining the biholomorphism part.}). In contrast, \( N_{\mathbb{C}} := G^{\mathbb{C}} / P^{\mathbb{C}} \) is a Hermitian symmetric space of compact type (see~\cite{Tak84}).

\subsection*{Holomorphic embedding}The inclusion \( K^{\mathbb{C}} \subset G^{\mathbb{C}} \) induces a holomorphic embedding \( TN \hookrightarrow N_{\mathbb{C}} \) as an open dense \( K^{\mathbb{C}} \)-orbit. Both \( TN \) and \( N_{\mathbb{C}} \) admit invariant symplectic structures: \( \mathrm{d}\lambda \) (the pullback of the canonical symplectic form on \( T^*N \) via the metric \( g \)) and \( \omega_{\mathrm{KKS}} \) (the Kirillov--Kostant--Souriau form, arising from the realization of \( N_{\mathbb{C}} \) as a coadjoint orbit). These symplectic forms, together with their respective complex structures, define Kähler structures. However, the embedding \( TN \hookrightarrow N_{\mathbb{C}} \) cannot be Kähler, as can already be seen from volume considerations:
\[
\mathrm{vol}(TN) = \infty > \mathrm{vol}(N_{\mathbb{C}}).
\]
This holomorphic embedding nevertheless shows that the tangent bundle of a symmetric \( R \)-space is uniruled; that is, there exists a (pseudo-)holomorphic curve through every point. Note that adapted complex structures are compatible with the symplectic form \( \mathrm{d}\lambda \), but the holomorphic curves in \( TN \) obtained in this way will have infinite energy. In symplectic topology, (pseudo-)holomorphic curves play a prominent role, and finite energy curves are generally better behaved than infinite energy ones. This is one of the motivations for seeking an open-dense symplectic embedding of a fiberwise convex neighborhood of the zero section of \( (TN, \mathrm{d}\lambda) \) into \( (N_{\mathbb{C}}, \omega_{\mathrm{KKS}}) \). Stretching the neck along the boundary of this fiberwise convex neighborhood could then provide finite-energy foliations of the tangent bundle \( TN \). In the cases \( N = \mathbb{S}^n, \mathbb{RP}^n, \mathbb{CP}^n \), or for Hermitian symmetric spaces of compact type, such symplectic embeddings were constructed explicitly (case by case) in~\cite{Ad22, Bim24.2, B24}, and used to compute symplectic capacities of disc tangent bundles \( (DTN, \mathrm{d}\lambda) \).

\subsection*{Symplectic embedding}In this article, we provide a systematic construction of such symplectic embeddings for all symmetric \( R \)-spaces. To state our main result, we introduce some notation. Let \( \mathfrak{k} = \mathfrak{h} \oplus \mathfrak{l} \) be the Cartan decomposition associated with the symmetric pair \( (K, H) \), and let \( \mathfrak{a} \subset \mathfrak{l} \) be a maximal abelian subalgebra. Note that \( \dim(\mathfrak{a}) =: \mathrm{rk}(N) \). These integrate into maximal flats, which can be thought of as immersed tori. Let \( \Sigma \) be the restricted root system of \( \mathfrak{k} \) with respect to \( \mathfrak{a}\subset \mathfrak{l} \), and define
\[
\Box_r := \big\{ X \in \mathfrak{a} \mid |\alpha(X)| < r \text{ for all } \alpha \in \Sigma \big\}.
\]
We equivariantly associate a fiberwise convex neighborhood of the zero section by
\[
U_r N := \big\{ (k, X) \mid X \in \Box_r \big\} \subset TN \cong K \times_H \mathfrak{l}.
\]
The set \( U_r N \subset TN \) is open, as \( H \) acts transitively on maximal flats. In the non-compact setting, such neighborhoods of the zero section are often referred to as Grauert domains (see, for example,~\cite{BHH03,AG90}), and for certain value of \( r \), they yield the largest domain on which the adapted invariant complex structure is defined.

\begin{theorem}[Corollary~\ref{symplectomorphism}]\label{thmintro}
For \( r \leq \frac{\mathrm{rk}(N_{\mathbb{C}})}{\mathrm{rk}(N)} \), there exists a \( K \)-equivariant symplectic embedding
\[
(U_r N, \mathrm{d}\lambda) \hookrightarrow (N_{\mathbb{C}}, \omega_{\mathrm{KKS}}),
\]
which is open-dense for \( r = \frac{\mathrm{rk}(N_{\mathbb{C}})}{\mathrm{rk}(N)} \).
\end{theorem}

A similar theorem was proved by Torres in~\cite[Thm.~1.1]{Tor13} in the broader setting of coadjoint orbits. However, in~\cite[Thm.~1.1]{Tor13}, the neighborhood of the zero section identified with the open-dense orbit is not explicit. The explicit characterization of the neighborhood \( U_r N \) is essential for computing capacities.

\subsection*{Capacities} We compute both the Gromov width and the Hofer--Zehnder capacity of the domains $(U_rN,\dd\lambda)$. Notably, as explained in Section~\ref{orbits at infinity}, either \( \mathrm{rk}(N_{\mathbb{C}}) = \mathrm{rk}(N) \) or \( \mathrm{rk}(N_{\mathbb{C}}) = 2 \cdot \mathrm{rk}(N) \) and \( \mathrm{rk}(N_{\mathbb{C}}) = 2 \cdot \mathrm{rk}(N) \) if and only if $N$ is simply connected.

\begin{theorem}
Let \( \mathrm{sys} \) denote the length of the shortest closed geodesic of the symmetric \( R \)-space \( (N, g) \). Then
\[
c_G(U_1 N, \mathrm{d}\lambda) = c_{HZ}(U_1 N, \mathrm{d}\lambda) =
\begin{cases}
\mathrm{sys}, & \text{if } \mathrm{rk}(N_{\mathbb{C}}) = 2 \cdot \mathrm{rk}(N), \\
2 \cdot \mathrm{sys}, & \text{if } \mathrm{rk}(N_{\mathbb{C}}) = \mathrm{rk}(N).
\end{cases}
\]
Moreover, if \( \mathrm{rk}(N) = 1 \), that is, if \( N \in \{ \mathbb{S}^n, \rp^n, \cp^n, \mathbb{H}\mathrm{P}^n, \mathbb{O}\mathrm{P}^2 \} \), then \( U_r N = D_r N \).
\end{theorem}

This theorem generalizes~\cite[Thm.~A]{Bim24.2}, where the Hofer--Zehnder capacity of \( (D_1 N, \mathrm{d}\lambda) \) was computed when \( N \in \{ \mathbb{RP}^n, \mathbb{CP}^n \} \). Note that both spaces \( \mathbb{RP}^n \) and \( \mathbb{CP}^n \) are symmetric \( R \)-spaces, and their complexifications are \( \mathbb{CP}^n \) and \( \mathbb{CP}^n \times \mathbb{CP}^n \), respectively. Observe that 
$$
\mathrm{rk}(\mathbb{CP}^n) = 1 = \mathrm{rk}(\mathbb{RP}^n) \quad \text{and}\quad \mathrm{rk}(\mathbb{CP}^n \times \mathbb{CP}^n) = 2 = 2 \cdot \mathrm{rk}(\mathbb{CP}^n).
$$

\subsection*{Disc tangent bundles}
Disc bundles \( D_r N := \big\{ (x, v) \mid |v|_x < r \big\} \) are of particular interest, as they arise as sublevel sets of the kinetic Hamiltonian \( E(x, v) = \frac{1}{2} |v|_x^2 \), which generates the geodesic flow, revealing a deep connection to Riemannian geometry. For simply connected \( N \), we obtain the following corollary:

\begin{corollary}\label{HZ1}
If \( N \) is simply connected,\footnote{\( N \) is simply connected if and only if \( \mathrm{rk}(N_{\mathbb{C}}) = 2 \cdot \mathrm{rk}(N) \); see Appendix~\ref{classification}.} then
$
c_{HZ}(D_1 N, \mathrm{d}\lambda) = \mathrm{sys}.
$
\end{corollary}

In the non-simply connected case, constructing a lower bound is more challenging and could potentially be achieved using billiards. However, the analogue of Corollary~\ref{HZ1} does not hold in general. For example, if
\[
N = Q_{p,q}(\mathbb{R}) := \bigg\{ [x] \in \mathbb{RP}^{p+q+1} \ \bigg| \ x_1^2 + \cdots + x_{p+1}^2 - x_{p+2}^2 - \cdots - x_{p+q+2}^2 = 0 \bigg\}
\]
is a real quadric, we find:

\begin{theorem}\label{quad}
If \( 1 \leq p \leq q \), then $
c_{HZ}(D_1 Q_{p,q}(\mathbb{R}), \mathrm{d}\lambda) = \sqrt{2} \cdot \mathrm{sys}.
$
\end{theorem}

On the other hand, for \( N = \mathbb{RP}^n \), we have \( U_r N = D_r N \), and therefore
\[
c_{HZ}(D_1 \mathbb{RP}^n, \mathrm{d}\lambda) = 2 \cdot \mathrm{sys}.
\]
In both cases the capacity is given by the length of the shortest closed contractible geodesic. It would be interesting to understand the Hofer--Zehnder capacity for the other non-simply connected symmetric \( R \)-spaces, but at present, we do not know how to approach this question systematically.

\subsection*{Finsler Metrics and Symplectic Systolic Inequalities}

An interesting observation is that fiberwise convex subsets of the tangent bundle correspond to Finsler metrics. The Finsler metric associated with the neighborhood \( U_r N \) is a type of \( L^\infty \)-norm. More precisely,
\[
F^p_x \colon T_xN \to \mathbb{R}_{\geq 0}, \quad v = a^\#_x \mapsto \Vert \mathrm{ad}_a \Vert_p,
\]
defines a \( K \)-equivariant Finsler metric on \( N \), where \( \Vert \cdot \Vert_p \) denotes the operator norm on \( \mathrm{End}(\mathfrak{k}) \) induced by the \( L^p \)-norm on \( \mathfrak{k} \). A straightforward computation shows that \( U_1 N \) is the unit disc tangent bundle with respect to \( F^\infty \).

When \( N \) is simply connected, modifying the function \( \mathrm{sys} \cdot F^1 \) yields an admissible Hamiltonian on \( D_1^{F^1}N \subset D_1^{F^\infty}N = U_1N \). Hence, the Hofer--Zehnder capacity of all \( F^p \)-unit disc bundles is the same. This implies that the invariant Riemannian metric \( g \), which induces \( F^2 \), does not optimize the symplectic systolic ratio for the Hofer--Zehnder capacity.

\subsection*{Outline}
We begin in Section~\ref{intro} with a brief introduction to symmetric \( R \)-spaces, summarizing the main results of~\cite{Tak84, Qu14}. In Section~\ref{symplectic embedding}, we prove Theorem~\ref{thmintro}. The proof is structured in four steps: in Section~\ref{orbits at infinity}, we describe the orbits at infinity \( \Delta := N_{\mathbb{C}} \setminus TN \); in Section~\ref{liouville}, we construct an invariant Liouville vector field on \( N_{\mathbb{C}} \setminus \Delta \); in Section~\ref{completion}, we define the symplectic completion of \( N_{\mathbb{C}} \setminus \Delta \); and in Section~\ref{momentum}, we compare momentum map images to complete the proof of Theorem~\ref{thmintro}.

Section~\ref{Hamcircle} discusses a Hamiltonian circle action on \( N_{\mathbb{C}} \) that facilitates the computation of symplectic capacities of \( (U_r N, \mathrm{d}\lambda) \) in Section~\ref{capacities UN}. We conclude with a partial discussion of the Hofer--Zehnder capacity of disc tangent bundles and the proof of Theorem~\ref{quad}.

\subsection*{Acknowledgment}

I am very grateful to Stéphanie Cupit-Foutou for her guidance and support, especially in navigating the literature on complex homogeneous geometry. I also thank Stefan Nemirovski for several helpful discussions.

This research was supported by the DFG-funded Collaborative Research Center CRC/TRR 191 \textit{Symplectic Structures in Geometry, Algebra, and Dynamics} (281071066).

\section{Symmetric $R$-spaces}\label{intro}

In this section, we adopt the standard notation \( (G, K) \) for symmetric pairs and write \( \mathfrak{g} = \mathfrak{k} \oplus \mathfrak{p} \) as the Cartan decomposition. Later, we will encounter three symmetric pairs (and corresponding Cartan decompositions), associated respectively with the symmetric \( R \)-space \( N \), its complexification \( N_{\mathbb{C}} \), and the anti-holomorphic involution on \( N_{\mathbb{C}} \). All relevant groups will be introduced carefully in Section~\ref{rspaces}.

\subsection{Symmetric spaces}
This introduction is intended to establish notation and provide context; it is not meant to be complete or self-contained. For a detailed exposition on symmetric spaces, we refer the reader to Helgason~\cite{HG01} and Wolf~\cite{W72}.

\subsection*{Geodesic symmetry} A Riemannian manifold \( (N, g) \) is called a \emph{symmetric space} if for every point \( p \in N \), there exists an isometry \( s_p: N \to N \) such that:
\begin{enumerate}
    \item \( s_p(p) = p \),
    \item The differential \( (ds_p)_p = -\mathrm{id} \) on \( T_pN \),
    \item \( s_p \) is an involution: \( s_p^2 = \mathrm{id} \).
\end{enumerate}
The map \( s_p \) is called the \emph{geodesic symmetry at \( p \)}. It reverses all geodesics through \( p \); that is, for any geodesic \( \gamma(t) \) with \( \gamma(0) = p \), we have
\[
s_p(\gamma(t)) = \gamma(-t).
\]

\subsection*{Cartan involution}Let \( G = \mathrm{Isom}(N)^0 \) be the identity component of the full isometry group of \( N \), and fix a base point \( o \in N \). Let \( K \subset G \) be the stabilizer of \( o \), so that \( N \cong G/K \). The geodesic symmetry \( s_o \) induces an involutive automorphism \( \sigma: G \to G \) defined by
\[
\sigma(g) = s_o \circ g \circ s_o^{-1}.
\]
This satisfies \( \sigma^2 = \mathrm{id} \), and its differential at the identity \( \theta = d\sigma_e \) is an involutive automorphism of the Lie algebra \( \mathfrak{g} = \mathrm{Lie}(G) \).

\subsection*{Cartan decomposition} The eigenspace decomposition of \( \mathfrak{k} \) under \( \theta \) yields the \emph{Cartan decomposition}:
\[
\mathfrak{g} = \mathfrak{k} \oplus \mathfrak{p},
\]
where:
\begin{itemize}
    \item \( \mathfrak{k} = \{ X \in \mathfrak{g} : \theta(X) = X \} \) is the Lie algebra of \( K \),
    \item \( \mathfrak{p} = \{ X \in \mathfrak{g} : \theta(X) = -X \} \) is identified with the tangent space \( T_oN \) via the canonical projection \( G \to G/K \).
\end{itemize}
The Lie bracket relations follow as $\theta$ is a Lie algebra automorphism:
\[
[\mathfrak{k}, \mathfrak{k}] \subset \mathfrak{k}, \quad
[\mathfrak{k}, \mathfrak{p}] \subset \mathfrak{p}, \quad
[\mathfrak{p}, \mathfrak{p}] \subset \mathfrak{k}.
\]

\subsection*{Irreducibility} Note, that the commutator relation $[\mathfrak{k}, \mathfrak{p}] \subset \mathfrak{p}$ implies that $K$ acts on $\mathfrak{p}$. A symmetric space is called \emph{irreducible} if this representation is irreducible, which implies that $(N,g)$ cannot be written as a Riemannian product of symmetric spaces.

\subsection*{Unique invariant metric} A \( G \)-invariant Riemannian metric on \( N \) is determined by an \( \operatorname{Ad}(K) \)-invariant inner product on \( \mathfrak{p} \). If the space is irreducible the (up to scaling) unique inner product is the Killing form. Hence, the Riemannian metric $g$ must be induced by the Killing form.

\subsection*{Duality} Irreducible symmetric spaces are assigned a type: \emph{compact} or \emph{non-compact}, depending on whether \( G \) is compact or non-compact. Each symmetric space of non-compact type has a unique \emph{compact dual} (and vice versa), obtained by complexification of $\mathfrak{g}$ and taking the compact real form with its associated Cartan decomposition.

\subsection*{Maximal flats \& rank} A central geometric concept in the theory of symmetric spaces is the notion of a \emph{maximal flat}. A maximal flat is a totally geodesic, flat submanifold of \( N \), associated with a maximal abelian subalgebra \( \mathfrak{a} \subset \mathfrak{p} \), that is,
\[
[X, Y] = 0 \quad \text{for all } X, Y \in \mathfrak{a},
\]
and \( \mathfrak{a} \) is maximal with respect to this property. The exponential image \( \exp(\mathfrak{a}) \cdot o \subset N \) defines a maximal flat submanifold. The \emph{rank} of the symmetric space \( N \) is defined as
\[
\operatorname{rank}(N) = \dim \mathfrak{a}.
\]

\subsection*{Restricted root system}

Let \( \mathfrak{a} \subset \mathfrak{p} \) be a maximal abelian subspace. The \emph{restricted root system} describes how \( \mathfrak{g} \) decomposes under the adjoint action of \( \mathfrak{a} \). These are also referred to as relative root systems; see~\cite[Ch.~VI.4]{Kna96} for details.

\vspace{1em}

For \( H \in \mathfrak{a} \), the operator \( \mathrm{ad}_H \) on $\mathfrak{g}$ is self-adjoint with respect to the inner product \( (\cdot, \theta \cdot) \), where \( (\cdot, \cdot) \) is the Killing form and \( \theta \) is the Cartan involution. As a result, the eigenvalues of \( \mathrm{ad}_H \) are real. Since \( \mathfrak{a} \) is abelian, we may define simultaneous eigenspaces:
\[
\mathfrak{g}_\alpha := \{ X \in \mathfrak{g} \mid [H, X] = \alpha(H) X \quad \text{for all } H \in \mathfrak{a} \}.
\]
The nonzero linear functionals \( \alpha \in \mathfrak{a}^* \) for which \( \mathfrak{g}_\alpha \neq 0 \) are called \emph{restricted roots}, and they form the restricted root system:
\[
\Sigma := \{ \alpha \in \mathfrak{a}^* \mid \mathfrak{g}_\alpha \neq 0 \}.
\]
This yields the orthogonal direct sum decomposition:
\[
\mathfrak{g} = \mathfrak{a} \oplus \mathfrak{m} \oplus \bigoplus_{\alpha \in \Sigma} \mathfrak{g}_\alpha,
\]
where \( \mathfrak{m} := Z_{\mathfrak{k}}(\mathfrak{a}) \) is the centralizer of \( \mathfrak{a} \) in \( \mathfrak{k} \).

These are called restricted root systems because they arise as restrictions of the root system of \( \mathfrak{g}^\mathbb{C} \). Specifically, if \( \mathfrak{a} \) is extended to a maximal abelian subalgebra \( \mathfrak{h} \subset \mathfrak{g} \), then its complexification \( \mathfrak{h}^\mathbb{C} \subset \mathfrak{g}^\mathbb{C} \) is a Cartan subalgebra. For each nonzero linear functional \( \beta \in (\mathfrak{h}^\mathbb{C})^* \), define the corresponding \emph{(absolute) root space}:
\[
\mathfrak{g}_\beta^\mathbb{C} := \{ X \in \mathfrak{g}^\mathbb{C} \mid [H, X] = \beta(H) X \quad \text{for all } H \in \mathfrak{h}^\mathbb{C} \}.
\]
Let \( \Delta \) denote the resulting root system. The restricted root spaces can be recovered by restriction:
\[
\mathfrak{g}_\alpha = \mathfrak{g} \cap \bigoplus_{\substack{\beta \in \Delta \\ \beta|_{\mathfrak{a}} = \alpha}} \mathfrak{g}_\beta^\mathbb{C}.
\]
If \( \mathfrak{g}^\mathbb{C} \) is semisimple, each root space \( \mathfrak{g}_\beta^\mathbb{C} \) is one-dimensional. This, however, need not be true for the restricted root spaces \( \mathfrak{g}_\alpha \).

\subsection{Hermitian symmetric spaces}

A Riemannian symmetric space \( N = G/K \) is called a \emph{Hermitian symmetric space} if it admits a \( G \)-invariant complex structure \( J \) that is compatible with the Riemannian metric, i.e., if \( (N, g, J) \) is a Kähler manifold.

\vspace{1em}

The following conditions are equivalent and characterize Hermitian symmetric spaces among Riemannian symmetric spaces:
\begin{enumerate}
    \item There exists an \( \operatorname{Ad}(K) \)-invariant complex structure \( J_o: \mathfrak{p} \to \mathfrak{p} \), i.e., a linear map satisfying \( J_o^2 = -\mathrm{id}_{\mathfrak{p}} \) and \( \operatorname{Ad}(k) \circ J_o = J_o \circ \operatorname{Ad}(k),\ \ \forall k\in K \).
    
    \item The center of \( \mathfrak{k} \) is nontrivial, i.e., \( \dim Z(\mathfrak{k}) \geq 1 \). In the irreducible case, \( \dim Z(\mathfrak{k}) = 1 \).
\end{enumerate}

To see that the second condition implies the first, observe that for any nonzero element \( Z \in Z(\mathfrak{k}) \), the map \( \operatorname{ad}_Z|_{\mathfrak{p}} \) is \( \operatorname{Ad}(K) \)-equivariant and skew-symmetric with respect to the Killing form. Hence, \( \operatorname{ad}_Z|_{\mathfrak{p}}^2 \) is self-adjoint with non-negative real eigenvalues. In the irreducible case, Schur's Lemma implies
\[
\operatorname{ad}_Z^2|_{\mathfrak{p}} = -\lambda^2 \cdot \mathrm{id}_{\mathfrak{p}}, \quad \lambda > 0.
\]
Therefore, the complex structure is (up to sign) uniquely defined by
\[
J_o := \frac{1}{\lambda} \operatorname{ad}_Z|_{\mathfrak{p}}.
\]

Moreover, the geodesic symmetries are holomorphic: the differential of the symmetry at the origin,
\[
(ds_o)_o : T_o N \cong \mathfrak{p} \to \mathfrak{p},
\]
is equal to \( -\mathrm{id}_{\mathfrak{p}} \), and it commutes with the complex structure \( J_o \).

\subsection*{Polyspheres and polydiscs}

Recall that the root spaces \( \mathfrak{g}^\mathbb{C}_\alpha \) for \( \alpha \in \Delta \) are one-dimensional. Therefore, the subspace \( \mathfrak{h}^\mathbb{C}_\alpha := [\mathfrak{g}^\mathbb{C}_\alpha, \mathfrak{g}^\mathbb{C}_{-\alpha}] \subset \mathfrak{h}^\mathbb{C} \) is also one-dimensional. There exists a unique element \( H_\alpha \in \mathfrak{h}_\alpha \) such that \( \alpha(H_\alpha) = 2 \). Moreover, it is easy to see that there exist elements \( X_\alpha \in \mathfrak{g}_\alpha \) and \( Y_\alpha \in \mathfrak{g}_{-\alpha} \) satisfying
\[
[H_\alpha, X_\alpha] = 2 X_\alpha, \quad [H_\alpha, Y_\alpha] = -2 Y_\alpha, \quad \text{and} \quad [X_\alpha, Y_\alpha] = H_\alpha.
\]
These elements generate a subalgebra of \( \mathfrak{g}^\mathbb{C} \) isomorphic to \( \mathfrak{sl}(2, \mathbb{C}) \), which we denote by \( \mathfrak{g}^\mathbb{C}[\alpha] \). Two roots \( \alpha, \beta \in \Delta \) are called \emph{strongly orthogonal} if \( \alpha \pm \beta \notin \Delta \). In this case, \( [\mathfrak{g}_\alpha^\mathbb{C}, \mathfrak{g}_\beta^\mathbb{C}] = 0 \). According to~\cite[Prop.~7.4, Ch.~VIII]{HG01}, there exist strongly orthogonal positive non-compact roots \( \gamma_1, \ldots, \gamma_r \), where \( r = \mathrm{rank}(N) \). The subalgebra
\[
\bigoplus_{i=1}^r \mathfrak{g}^\mathbb{C}[\gamma_i] \subset \mathfrak{g}^\mathbb{C}
\]
is isomorphic to \( \mathfrak{sl}(2, \mathbb{C})^r \). Its intersection with \( \mathfrak{g} \) yields a real subalgebra isomorphic to either \( \mathfrak{su}(2)^r \) or \( \mathfrak{sl}(2, \mathbb{R})^r \), depending on whether \( N \) is of compact or non-compact type.
Intersecting these subalgebras with \( \mathfrak{p} \) gives Lie triple systems, which correspond to totally geodesically embedded copies of either polyspheres \( (\mathbb{C}P^1)^r \) or polydiscs \( (\mathbb{C}\mathrm{H}^1)^r \), respectively.

\subsection*{Realization as a (co)adjoint orbit}

Let \( N = G / K \) be an irreducible Hermitian symmetric space of compact or noncompact type, and fix the element \( Z \in Z(\mathfrak{g}) \) in the center of \( \mathfrak{g} \) that induces the complex structure \( J \). The adjoint orbit through \( Z \),
\[
\mathcal{O}_Z := \operatorname{Ad}(G) \cdot Z \subset \mathfrak{g},
\]
is isomorphic to the symmetric space \( N = G / K \), since the stabilizer of \( Z \) in \( G \) is exactly \( K \).
Using the Killing form to identify \( \mathfrak{g} \cong \mathfrak{g}^* \), this adjoint orbit also corresponds to a coadjoint orbit. The associated Kirillov–Kostant–Souriau (KKS) symplectic form combines with the invariant metric and complex structure to define a \( G \)-invariant Kähler structure on \( N \). For \( v, w \in T_x N \), this structure is given by:
\[
g_x(v, w) = (v, w), \qquad J_x(v) = \operatorname{ad}_x(v), \qquad \omega_x(v, w) = (x, [v, w]),
\]
where \( (\cdot, \cdot) \) denotes the inner product induced by the Killing form.
Note that any vector \( v \in T_x N \) can be written as
\[
v = (a^\#)_x = [x, v] \in T_x N \cong [x, \mathfrak{g}] \subset \mathfrak{g}.
\]
Thus, every irreducible Hermitian symmetric space arises as a \( G \)-equivariant Kähler manifold via the (co)adjoint orbit of the central element \( Z \in Z(\mathfrak{g}) \) that defines the complex structure.

\subsection{Symmetric \( R \)-spaces}\label{rspaces}

We now introduce the notation that will be used throughout the remainder of this paper. Since we study a symmetric space \( N \) embedded in a Hermitian symmetric space \( N_{\mathbb{C}} \), several distinct Lie groups naturally appear in the discussion. To assist the reader, we provide a reference table (Table~\ref{groups}) summarizing the relevant groups and their roles.

\vspace{1em}

Let \( N \cong K/H \) be an indecomposable symmetric space of compact type, associated with a Riemannian symmetric pair \( (K, H) \). Let \( g \) denote the invariant Riemannian metric on \( N \) such that \( K \) is the identity component of the isometry group. We say that \( N \) is a \emph{symmetric \( R \)-space} if there exists a non-compact Lie group \( G \supset K \) acting on $N$ and a parabolic subgroup \( P \subset G \) such that \( N \cong G/P \). Symmetric \( R \)-spaces have been classified (see \cite{KN64}); a complete list is provided in Appendix~\ref{classification}. The classification includes a variety of important examples, such as real, complex, and quaternionic Grassmannians, certain compact Lie groups, and real and complex projective quadrics. A central fact is that \emph{every Hermitian symmetric space is a symmetric \( R \)-space}, where \( G \) is taken to be the biholomorphism group. More precisely, the indecomposable symmetric \( R \)-spaces fall into two mutually exclusive classes (cf.\ \cite{Tak84}):
\begin{enumerate}
    \item \emph{Irreducible Hermitian symmetric spaces of compact type}
    \item \emph{Indecomposable symmetric \( R \)-spaces of non-Hermitian type}
\end{enumerate}

Symmetric \( R \)-spaces naturally arise as \emph{real forms} of Hermitian symmetric spaces of compact type. That is, they are realized as the fixed point set of an anti-holomorphic involution \( \tau \) acting on a Hermitian symmetric space \( N_\mathbb{C} \). If \( N \) is itself Hermitian symmetric this is easy to see: given any anti-holomorphic isometry \( f \) of \( N \), we define an involution  
\[
\tau: N \times N \to N \times N, \quad (a,b) \mapsto (f^{-1}(b), f(a)).
\]
The fixed point set of \( \tau \) is  
\[
\operatorname{Fix}(\tau) = \{(a, f(a)) \in N \times N\} \cong N,
\]
offering an explicit realization of \( N \) as a real form of \( N \times N \). A fundamental result by Takeuchi \cite{Tak84}  characterizes symmetric \( R \)-spaces of non-Hermitian type as real forms of irreducible Hermitian symmetric spaces:

\begin{theorem}[\cite{Tak84}]\label{hss}
    Every indecomposable non-Hermitian symmetric \( R \)-space is a real form of an irreducible Hermitian symmetric space of compact type, and vice versa.
\end{theorem}

This correspondence is established by realizing \( N \cong G/P \) as a real form of the complex space \( N_\C := G^\C / P^\C \), where \( G^\C \) and \( P^\C \) denote the complexifications of \( G \) and \( P \), respectively. The space \( N_\C \) is then a Hermitian symmetric space with biholomorphism group \( G^\C \), and isometry group \( G^\vee \subset G^\C \), where \( G^\vee \) is the connected Lie subgroup whose Lie algebra is the compact real form of \( \mathfrak{g}^\C \). In particular, we may identify  
\[
N_\C \cong G^\vee / K^\vee,
\]  
where \( K^\vee \subset G^\vee \) denotes the stabilizer of a point. In the special case where \( N \) is itself Hermitian symmetric and \( G = \mathrm{Aut}(N) \) is a complex Lie group, we have \( G^\C = G \times G \), which implies  
\[
N_\C = N \times N.
\]

\subsubsection{Table with groups}\label{groups} There are too many groups involved. The following table summarizes them.

\begin{center}
\begin{tabular}{|c|p{8cm}|p{5cm}|}
\hline
\textbf{Group} & \textbf{Description} & \textbf{Related To} \\
\hline
$K$ & Compact Lie group (identity component of isometry group of $N$) & $N = K/H$, $K\subset G$, $K\subset G^\vee$ \\
\hline
$H$ & Closed subgroup of $K$; stabilizer subgroup for symmetric space $N$ & $N = K/H$ \\
\hline
$G$ & Non-compact Lie group acting transitively on $N$ ($K\subset G$ maximal compact) &  $N = G/P$, $G \supset K$ \\
\hline
$P$ & Parabolic subgroup of $G$ & $N = G/P$ \\
\hline
$G^\vee$ & Compact real form of $G^\mathbb{C}$; identity component of isometry group of $N_\mathbb{C}$ & $N_\mathbb{C} = G^\vee / K^\vee$,\ \  $G^\vee \subset G^\mathbb{C}$  compact dual of $G$\\
\hline
$K^\vee$ & Stabilizer subgroup in $G^\vee$ & $N_\mathbb{C} = G^\vee / K^\vee$ \\
\hline
\end{tabular}
\end{center}

It might be instructive to look at an example. For this reason we included the example $N=\mathrm{S}^n$ in Appendix \ref{spheres}.

\subsection{Geometric interpretation}

There is a geometric realization of symmetric \( R \)-spaces in terms of (co)adjoint orbits. Note that \( \mathfrak{g}^\vee \) is equipped with two commuting Cartan involutions: \( \theta \) and \( \sigma \), corresponding to the symmetric pairs \( (G^\vee, K^\vee) \) and \( (G^\vee, K) \), respectively. These give rise to the Cartan decompositions
\[
\mathfrak{g}^\vee = \mathfrak{k}^\vee \oplus \mathfrak{p}^\vee \quad \text{(with respect to } \theta\text{),} \qquad
\mathfrak{g}^\vee = \mathfrak{k} \oplus \mathfrak{p} \quad \text{(with respect to } \sigma\text{)}.
\]
The symmetric space associated with the pair \( (G^\vee, K) \) is the one studied in~\cite{Qu14}. It is shown there that there exists an \emph{extrinsically symmetric} element \( \xi \in \mathfrak{p} \) such that
\begin{equation}\label{adjoint orbit}
    \mathfrak{p} \supset \operatorname{Ad}_{G^\vee}(K) \cdot \xi \cong N \hookrightarrow N_{\mathbb{C}} \cong \operatorname{Ad}_{G^\vee}(G^\vee) \cdot \xi \subset \mathfrak{g}^\vee.
\end{equation}
Moreover, the anti-holomorphic involution \( \tau \) fixing \( N \subset N_{\mathbb{C}} \) is given by the restriction of \( -\sigma \) to \( N_{\mathbb{C}} \subset \mathfrak{g}^\vee \).

\vspace{1em}

This realization provides explicit formulas for the momentum maps associated with the Hamiltonian \( K \)-actions on both \( TN \) and \( N_{\mathbb{C}} \). The embedding \( N \hookrightarrow \mathfrak{p} \subset \mathfrak{g}^\vee \) as an adjoint suborbit identifies \( TN \) with a subset of \( \mathfrak{p} \times \mathfrak{p} \), so that for \( (x, v) \in TN \), the bracket \( [x, v] \in \mathfrak{k} \). This yields the following lemma:

\begin{lemma}
    The \( K \)-actions on \( TN \) and \( N_{\mathbb{C}} \) are Hamiltonian with \( K \)-equivariant momentum maps given by
    \begin{equation*}
    \begin{split}
        &\mu_{TN} \colon TN \to \mathfrak{k}, \quad (x, v) \mapsto [x, v], \\
        &\mu_{N_{\mathbb{C}}} \colon N_{\mathbb{C}} \to \mathfrak{k}, \quad a \mapsto \operatorname{pr}_{\mathfrak{k}}(a),
    \end{split}
    \end{equation*}
    where \( \operatorname{pr}_{\mathfrak{k}} \) denotes the orthogonal projection with respect to the Killing form, and \( \mathfrak{k} \cong \mathfrak{k}^* \) via the same identification.
\end{lemma}

\begin{proof}
    The formula for \( \mu_{TN} \) follows from the inclusion \( [\mathfrak{p}, \mathfrak{p}] \subset \mathfrak{k} \) and the fact that the bracket map
    \[
    [\cdot, \cdot] \colon TN \subset \mathfrak{g}^\vee \times \mathfrak{g}^\vee \to \mathfrak{g}^\vee \cong (\mathfrak{g}^\vee)^*
    \]
    is a momentum map for any (co)adjoint orbit. Indeed, for all \( a \in \mathfrak{g}^\vee \),
    \[
    \mathrm{d}([x, v], a) = -\mathrm{d}(v, [x, a]) = -\mathrm{d}(\lambda(a^\#_x)) = \iota_{a^\#_x} \mathrm{d}\lambda,
    \]
    where \( a^\#_x = [x, a] \) is the fundamental vector field associated with \( a \in\mathfrak{g}^\vee\), and we use \( G^\vee \)-invariance of the Liouville form \( \lambda \) for the final equality.

    The formula for \( \mu_{N_{\mathbb{C}}} \) is simply the restriction of the standard momentum map for the coadjoint orbit \( \operatorname{Ad}_{G^\vee}(G^\vee) \cdot \xi \hookrightarrow \mathfrak{g}^\vee \cong (\mathfrak{g}^\vee)^* \) (see~\cite[Ch.~1]{Kr04}).
\end{proof}

\section{Symplectic embedding}\label{symplectic embedding}

\noindent
In this section, we prove Theorem~\ref{thmintro}. The proof is organized into the following steps:
\begin{enumerate}
    \item An explicit description of the boundary set at infinity, \( \Delta := N_{\mathbb{C}} \setminus TN \),
    \item Construction of an invariant Liouville vector field on \( N_{\mathbb{C}} \setminus \Delta \),
    \item Definition of the symplectic completion of \( N_{\mathbb{C}} \setminus \Delta \) and its symplectic identification with \( TN \),
    \item Comparison of momentum map images to conclude the proof of Theorem~\ref{thmintro}.
\end{enumerate}

\subsection{The Orbits at Infinity}\label{orbits at infinity}

The points "at infinity" are those not in the image of the holomorphic embedding \( TN \hookrightarrow N_{\mathbb{C}} \). More precisely, define
\[
\Delta := \left\{ p \in N_{\mathbb{C}} \mid \tau(p) \in \mathrm{Cut}_{N_{\mathbb{C}}}(p) \right\},
\]
where \( \mathrm{Cut}_{N_{\mathbb{C}}}(p) \) denotes the cut locus of \( p \) in \( N_{\mathbb{C}} \). The following two examples illustrate this construction and essentially suffice to justify the identification:
\[
(TN, i) \cong (N_{\mathbb{C}} \setminus \Delta, i).
\]

\begin{example}\label{minimal example 1}
Let \( N = \mathrm{S}^1 \), so that \( TN \cong \mathbb{C}^* \) and \( N_{\mathbb{C}} \cong \mathbb{C}P^1 \). The involution \( \tau \colon \mathbb{C}P^1 \to \mathbb{C}P^1 \) is given by \( \tau(z) = \bar{z}^{-1} \). It is easy to verify that
\[
\Delta = \{0, \infty\},
\]
and therefore,
\[
(TN, i) \cong (\mathbb{C}^*, i) \cong (\mathbb{C}P^1 \setminus \{0, \infty\}, i) \cong (N_{\mathbb{C}} \setminus \Delta, i),
\]
as claimed.
\end{example}

\begin{example}\label{minimal example 2}
A slightly more involved example is the case \( N = \mathbb{C}P^1 \), so that \( N_{\mathbb{C}} = \mathbb{C}P^1 \times \mathbb{C}P^1 \). The involution
\[
\tau \colon \mathbb{C}P^1 \times \mathbb{C}P^1 \to \mathbb{C}P^1 \times \mathbb{C}P^1, \quad \tau(a, b) = (-\bar{b}^{-1}, -\bar{a}^{-1}),
\]
is anti-holomorphic and fixes the Lagrangian copy \( \mathbb{C}P^1 \cong \{ (a, -a) \} \subset N_{\mathbb{C}} \). The diagonal \( \mathrm{SL}(2, \mathbb{C}) \)-action has two orbits:
\begin{itemize}
    \item an open orbit 
    \[
    \{ (a, b) \in \mathbb{C}P^1 \times \mathbb{C}P^1 \mid a \neq b \} \cong \mathrm{SL}(2, \mathbb{C}) / \mathbb{C}^* \cong \mathrm{SU}(2)^{\mathbb{C}} / \mathrm{S}(\mathrm{U}(1) \times \mathrm{U}(1))^{\mathbb{C}} \cong T\mathbb{C}P^1,
    \]
    \item and a closed orbit
    \[
    \{ (a, b) \in \mathbb{C}P^1 \times \mathbb{C}P^1 \mid a = b \} = \{ (a, b) \mid \tau(a, b) \in \mathrm{Cut}(a, b) \} =: \Delta.
    \]
\end{itemize}
Hence, we again have the identification
\[
(TN, i) \cong (\mathbb{C}P^1 \times \mathbb{C}P^1 \setminus \Delta, i).
\]
\end{example}

Comparing the ranks of \( N \) and \( N_{\mathbb{C}} \), there are only two possibilities (cf.~\cite[p.~301]{Tak84}): either \( \mathrm{rk}(N) = \mathrm{rk}(N_{\mathbb{C}}) \) or \( 2 \cdot \mathrm{rk}(N) = \mathrm{rk}(N_{\mathbb{C}}) \). This dichotomy can be explained by considering polyspheres preserved under the involution \( \tau \). Extend a maximal abelian subspace \( \mathfrak{a} \subset \mathfrak{l} \cong T_o N \) to a maximal abelian subspace \( \bar{\mathfrak{a}} \subset \mathfrak{p}^\vee \). The Lie triple system \( \bar{\mathfrak{a}} \oplus J_o \bar{\mathfrak{a}} \) integrates to a polysphere \( P = (\mathbb{C}P^1)^{\mathrm{rk}(N_{\mathbb{C}})} \) through \( o \), which is preserved by \( \tau \), meaning that \( \tau \) restricts to an anti-holomorphic involution on \( P \). There are only two anti-holomorphic involutions on \( \mathbb{C}P^1 \): one with real locus \( \mathbb{R}P^1 \), and one without fixed points. The first corresponds to Example~\ref{minimal example 1}, and the second induces an anti-holomorphic involution on \( \mathbb{C}P^1 \times \mathbb{C}P^1 \) with real locus \( \mathbb{C}P^1 \), as discussed in Example~\ref{minimal example 2}. Since the Weyl group acts transitively on the factors of \( P \), we cannot distinguish them individually. Therefore, the restriction of \( \tau \) to \( P \) must be a product of either only the first type or only the second. Moreover, since \( H \) acts transitively on maximal flats through the origin, the \( K \)-orbit of this polysphere covers \( N_{\mathbb{C}} \). It follows that the points not in the image of the holomorphic embedding \( (TN, i) \hookrightarrow (N_{\mathbb{C}}, i) \) are precisely the \( K \)-orbit of those points in \( P \). This yields the following proposition.

\begin{proposition}
The complement of the holomorphic embedding \( (TN, i) \hookrightarrow (N_{\mathbb{C}}, i) \) is given by \( \Delta = \Delta(N_{\mathbb{C}}) \), where
\[
\Delta(N_{\mathbb{C}}) = K \cdot \Delta(P).
\]
In particular, \( \Delta(N_{\mathbb{C}}) \) is a finite union of complex hyperplanes intersecting transversely.
\end{proposition}

\subsection{An Invariant Liouville Vector Field}\label{liouville}

Weinstein's tubular neighborhood theorem, applied to the inclusion \( N \subset N_{\mathbb{C}} \), yields a symplectic embedding of a small disc bundle:
\[
(D_\varepsilon N, \mathrm{d}\lambda) \hookrightarrow (N_{\mathbb{C}}, \omega_{\mathrm{KKS}}).
\]
On disc bundles, we have a natural invariant Liouville vector field, given by
\[
(Y_{TN})_{(x, v)} = (v)_x^{\mathcal{V}},
\]
which generates fiberwise scaling.\footnote{Here, the vertical lift \( T_x N \to T_{(x,v)} TN; \, w \mapsto (w)_x^{\mathcal{V}} \) refers to the canonical identification of tangent vectors with vertical vectors in the tangent bundle.} To globalize the Weinstein neighborhood, we seek a \( K \)-invariant Liouville vector field \( Y_{N_{\mathbb{C}}} \) on \( N_{\mathbb{C}} \setminus \Delta \) that matches \( Y_{TN} \) under the local identification.

\begin{lemma}
The symplectic form \( \omega_{\mathrm{KKS}} \) is exact on \( N_{\mathbb{C}} \setminus \Delta \).
\end{lemma}

\begin{proof}
By Theorem~\ref{complex identification tangent bundle}, the complement \( N_{\mathbb{C}} \setminus \Delta \) retracts onto the Lagrangian submanifold \( N \). Hence, $\omega_{KKS}$ vanishes on all cycles not intersecting $\Delta$ and therefore must be exact.
\end{proof}

Let \( \eta \in \Omega^1(N_{\mathbb{C}} \setminus \Delta) \) be a \( K \)-invariant primitive of \( \omega_{\mathrm{KKS}} \); such a form can be constructed by averaging any primitive over the compact group \( K \).

\begin{lemma}
We may assume that \( \eta \vert_N = 0 \).
\end{lemma}

\begin{proof}
As $\tau^*\omega_{KKS}=-\omega_{KKS}$, we have $\tau^*\eta=-\eta+\dd f$ for a $K$-invariant function $f:N_\C\setminus \Delta\to \R$. It follows from $\tau^2=\mathrm{id}$ that $\tau^*\dd f=-\dd f$, hence $\Tilde\eta:=\eta+\dd f/2$ satisfies $\tau^*\Tilde \eta =-\Tilde \eta$ and therefore $\Tilde\eta\vert_N=0$. The new $\Tilde{\eta}$ is another $K$-invariant primitive, we may therefore replace $\eta$. 
\end{proof}

Define a vector field $Y_{N_\C}$ on $N_\C\setminus\Delta$ implicitly via $\iota_{Y_{N_\C}}\omega_{KKS}=\eta$. It follows that $Y_{N_\C}$ is $K$-invariant and $(\Phi_{Y_{N_\C}}^t)^*\omega_{KKS}=e^t\omega_{KKS}$, where it is defined.
\begin{lemma}\label{Euler}
The invariant Liouville vector field satisfies:
    $$\mathrm{pr}_{\mathfrak{k}}(Y_{N_\C}(x))=\mathrm{pr}_{\mathfrak{k}}(x).$$
\end{lemma}
\begin{proof}
    We prove that $x\mapsto \mathrm{pr}_{\mathfrak{k}}(Y_{N_\C}(x))$ is a momentum map and vanishes on $N$. The vanishing follows directly from vanishing of $\eta$ on $N$. To prove the momentum map condition we compute for any $\xi\in \mathfrak{k}$:
\[
\langle \mathrm{pr}_{\mathfrak{k}}(Y_{N_\C}(x)),\xi\rangle =\langle Y_{N_\C}(x),\xi\rangle = \langle -J_x^2 (Y_{N_\C}(x)), \xi\rangle = -\langle [x,[x,Y_{N_\C}(x)]], \xi\rangle
\]
\[
=\langle [x,Y_{N_\C}(x)],[x,\xi]\rangle
= \langle x , [Y_{N_\C}(x),\xi^\#]\rangle=\omega_{KKS}(Y_{N_\C},\xi^\#)=\eta(\xi^\#).
\]
Further, 
$$
0=\mathcal{L}_{\xi^\#}\eta=\iota_{\xi^\#}\omega+\dd(\iota_{\xi^\#}(\eta)).
$$
Both together imply
$$
\dd \langle \mathrm{pr}_{\mathfrak{k}}(Y_{N_\C}(x)),\xi\rangle = -\iota_{\xi^\#}\omega.
$$
\end{proof}
Note, that $Y_{N_\C}$ is the desired invariant Liouville vector field. With its help we will in the next section define the symplectic completion of $N_\C\setminus \Delta$ and prove that it is symplectomorph to $(TN,\dd\lambda)$.

\subsection{The symplectic completion of $N_\C\setminus \Delta$}\label{completion} We will define the symplectic completion of $N_\C\setminus \Delta$ as completion of a Liouville domain with positive contact type boundary arbitrarily close to $\Delta$. We define the Liouville domain as preimage under the momentum map $\mu_{N_\C}$ of a star shaped neighborhood of $0\in \mathfrak{k}$. So we first need to understand the image of the momentum map.

\begin{lemma}\label{momentum images}
    We have 
    $$
    \mu_{N_\C}(N_\C\setminus\Delta)=\mathrm{Ad}_K\cdot\Box_{r_0};\ \ r_0=\frac{\mathrm{rk}(N_\C)}{\mathrm{rk}(N)}\in \lbrace 1,2 \rbrace.
    $$
\end{lemma}
\begin{proof}
    Per definition: 
    \begin{equation}\label{cut}
        p\in\Delta\Leftrightarrow \tau(p)\in \mathrm{Cut}(p)
    \end{equation}
    We may write $p=\exp_x(J_xv)$ for some $(x,v)\in TN$. As everything is $K$-equivariant we may assume $x=o$ and $J_o v\in\bar{\mathfrak{a}}$, for a maximal abelian subspace $\bar{\mathfrak{a}}\subset\mathfrak{p}^\vee$. Equation \eqref{cut} is now equivalent to
    $$
    \exists \alpha\in \bar\Sigma: \ \ \vert\alpha(v)\vert=\frac{\pi}{2},
    $$
    where $\bar\Sigma$ is the restricted root system of $(\mathfrak{g}^\vee,\bar{\mathfrak{a}})$. Viewing this as adjoint orbit and using that we normalized $o\equiv Z$ such that $J_o=\mathrm{ad}_Z$ this is equivalent to
    $$
     \exists \alpha\in \bar\Sigma: \ \ \vert\alpha(\mathrm{pr}_{\bar{\mathfrak{a}}}(p))\vert=1.
    $$
    If $\mathrm{rk}(N_\C)=\mathrm{rk}(N)$, we may choose $\bar{\mathfrak{a}}\subset\mathfrak{k}\cap\mathfrak{p}^\vee=\mathfrak{l}$, so that $\mathfrak{a}:=\bar{\mathfrak{a}}$ defines a maximal abelian subspace in $\mathfrak{l}$ and $\Sigma=\bar\Sigma$ defines a restricted root system for $(\mathfrak{k},\mathfrak{a}\subset\mathfrak{l})$. If $\mathrm{rk}(N_\C)=2\mathrm{rk}(N)$, we define a maximal abelian subspace of $\mathfrak{l}=\mathfrak{k}\cap\mathfrak{p}^\vee$ as
    $$
    \mathfrak{a}:=\lbrace v+\tau(v)\mid v\in\bar{\mathfrak{a}}\rbrace=\mathrm{Fix}(\tau\vert_{\bar{\mathfrak{a}}}).
    $$
    Note that $\bar{\mathfrak{a}}$ is $\tau$ invariant, hence $\mathrm{dim}\bar{\mathfrak{a}}=2\mathrm{dim}\mathfrak{a}$. We obtain the restricted root system for $(\mathfrak{k},\mathfrak{a}\subset\mathfrak{l})$ from $\bar\Sigma$ as
    $$
    \Sigma:=\lbrace \alpha+\alpha\circ\tau \mid \alpha\in\bar\Sigma\rbrace.  
    $$
    It follows that $p\in \Delta$, if 
    $$
    \exists\beta\in\Sigma:\ \ \vert \beta (\mathrm{pr}_{\mathfrak{a}}(p))\vert=2.
    $$
    Using $K$-equivariance of the momentum map and the fact that $H$ acts transitively on maximal flats in $N$, we obtain the claim.
\end{proof}

Let $B_\varepsilon\subset \Box_r\subset \mathfrak{a}$ be a star-shaped domain with smooth boundary $\varepsilon$-close to the boundary of $\Box_r$, i.e. $B_\varepsilon$ fills $\Box_r$ for $\varepsilon\to 0$.
\begin{lemma}
    The pre-image 
    $$
    W_\varepsilon:=\mu_{N_\C}^{-1}( (\mathrm{Ad}_K\cdot B_\varepsilon))
    $$ 
    has boundary of positive contact type.
\end{lemma}
\begin{proof}
    Clearly $\mathrm{Ad}_K\cdot B_\varepsilon$ is star shaped, hence the Euler vectorfield is transversal to the boundary. Using Lemma \ref{Euler} we see that the Euler vectorfield pulls back to the invariant Liouville vector field $Y_{N_\C}$ constructed in the previous section. As $Y_{N_\C}$ is everywhere transverse to $\partial W_\varepsilon$ and pointing outwards, $\partial W_\varepsilon$ is of positive contact type.\\
\end{proof}
We now define the symplectic completion $(\widehat{N_\C\setminus \Delta}, \hat\omega_{KKS})$ of $(N_\C\setminus \Delta, \omega_{KKS})$ to be the symplectic completion of the Liouville domain $(W_\varepsilon,\dd\eta)$. Indeed this does not depend on the choice of $B_\varepsilon$, as any other choice $\tilde{B}_\varepsilon$ determines a section in the normal bundle of $\partial B_\varepsilon$ identifying the boundaries of $\partial B_\varepsilon$ and $\partial\tilde B_\varepsilon$. Shifting by this section defines a symplectomorphism of the completions.

\vspace{1em}

Observe that as $\partial W_\varepsilon$ is $K$-invariant, $K$ also acts symplectically on $(\widehat{N_\C\setminus\Delta},\hat \omega_{KKS})$. The action is also Hamiltonian, where we extend the momentum map to the cylindrical end via
$$
\hat\mu_{N_\C}(r,x)=e^r\mu_{N_\C}(x).
$$
\begin{theorem}\label{symplectic completions}
    There is a $K$-equivariant symplectomorphism 
    $$
    \hat F:(\widehat{N_\C\setminus\Delta},\hat \omega_{KKS})\to(TN,\dd\lambda)
    $$
    intertwining the momentum maps, i.e. $\hat\mu_{N_\C}=\mu_{TN}\circ \hat F$.
\end{theorem}
\begin{proof}
    To prove this theorem we essentially copy the proof of \cite[Prop.\ 2.8]{KT05} and adapt it to our setup. As both momentum maps are proper, we find an open neighborhood $V\subset\mathfrak{t}$ of $0$ such that by Weinstein's tubular neighborhood theorem we have a $K$-equivariant symplectomorphism
    $$
    F:(\hat\mu^{-1}_{N_\C}(V),\dd\eta)\to (\mu^{-1}_{TN}(V),\dd\lambda).
    $$
    Up to shifting by a constant if needed, we can assume the momentum maps to be intertwined. Observe that Lemma \ref{Euler} tells us that the Liouville vector field $Y_{N_\C}$ is mapped to the Euler vector field of $\mathfrak{k}$ by $\hat\mu_{N_\C}$. Hence, $\hat\mu_{N_\C}\circ \Phi_{Y_{N_\C}}^t=e^t\hat\mu_{N_\C}$. Similarly $Y_{TN}(x,v):=(v)_x^\mathcal{V}$ defines a Liouville vector field on $TN$ satisfying $\mu_{TN}\circ \Phi_{Y_{TN}}^t=e^t\mu_{TN}$. 
    Now,
    $$
    \Tilde{F}_t: \Phi^{t}_{Y_{TN}}\circ F\circ \Phi^{-t}_{Y_{N_\C}}: (\hat\mu_{N_\C}^{-1}(e^tV),\dd\eta)\to (\mu_{TN}^{-1}(e^t V),\dd\lambda)
    $$
    is a $K$-equivariant symplectomorphism intertwining the momentum maps for any $t>0$ and thus induces the desired $\hat F$.
\end{proof}

\subsection{Comparing momentum map images}\label{momentum}
The symplectomorphism $\hat F$ of Theorem \ref{symplectic completions} intertwines the momentum maps. Therefore preimages of the momentum maps are identified symplectically. All that is left to do is to check that the momentum map images of $(N_\C\setminus\Delta)$ and $(U_{r_{\max}}N,\dd\lambda)$ for $r_{\max}=\frac{\mathrm{rk}(N_\C)}{\mathrm{rk}(N)}$ agree. This however follows immediately from Lemma \ref{momentum images}, as 
\begin{equation*}
    \begin{split}
        (x,v)\in U_{r_{\max}}N &\Leftrightarrow \exists k\in K, X\in\mathfrak{a}, \forall \alpha\in\Sigma:  (x,v)=\mathrm{Ad}_k(o,X)\ \&\ \vert \alpha(X)\vert < r_{\max}\\
        &\Leftrightarrow \mu_{TN}(x,v)=[x,v]\in \mathrm{Ad}_K\cdot\Box_{r_{\max}}.
    \end{split}
\end{equation*}
Which proves:
\begin{corollary}\label{symplectomorphism}
    We have the following symplectic identification
    $$
    (U_{r_{\max}}N,\dd\lambda)\cong (N_\C\setminus\Delta,\omega_{KKS}),\ \ r_{\max}=\frac{\mathrm{rk}(N_\C)}{\mathrm{rk}(N)}\in\lbrace 1,2 \rbrace. 
    $$
\end{corollary}

\section{Computing capacities}
In this section, we compute the Gromov width and Hofer--Zehnder capacity of \( (U_r N, \mathrm{d}\lambda) \). We begin by describing a Hamiltonian circle action on \( (N_{\mathbb{C}}, \omega_{\mathrm{KKS}}) \), which plays a central role in the subsequent computations.

\subsection{Embedding balls via Hamiltonian circle yctions}\label{Hamcircle}

As before, we realize \( N_{\mathbb{C}} \) as an adjoint orbit:
\[
N_{\mathbb{C}} \cong \operatorname{Ad}_{G^\vee}(K^\vee) \cdot \xi =: \mathcal{O}_\xi \subset \mathfrak{g}^\vee.
\]
This presentation yields a natural Hamiltonian circle action on \( (\mathcal{O}_\xi, \omega_{\mathrm{KKS}}) \), induced by the function
\[
H_{\xi} \colon \mathcal{O}_\xi \to \mathbb{R}, \quad a \mapsto 2\pi \langle \xi, a \rangle.
\]
The prefactor \( 2\pi \) ensures that the resulting flow has period 1. This normalization is consistent with \( \omega_{\mathrm{KKS}}(A) = 4\pi \) for any generator \( A \in H_2(N_{\mathbb{C}}, \mathbb{Z}) \) and $J_\xi=\mathrm{ad}_\xi$.\footnote{Indeed, \( H_2(N, \mathbb{Z}) \cong \mathbb{Z} \) for any irreducible Hermitian symmetric space; see Appendix~\ref{agenerator}.}

The Gromov width and Hofer--Zehnder capacity of Hermitian symmetric spaces of compact type have already been computed in previous works. See~\cite{LMZ13} for the Gromov width and~\cite{Cas20} for the Hofer--Zehnder capacity of more general coadjoint orbits.

\begin{theorem}[\cite{LMZ13, Cas20}]\label{capacities HSS}
Let \( (N_{\mathbb{C}}, \omega_{\mathrm{KKS}}) \) be a Hermitian symmetric space of compact type, normalized such that \( \omega_{\mathrm{KKS}}(A) = 4\pi \) on generators \( A \in H_2(N_{\mathbb{C}}, \mathbb{Z}) \). Then:
\[
c_G(N_{\mathbb{C}}, \omega_{\mathrm{KKS}}) = 4\pi \quad \text{and} \quad c_{HZ}(N_{\mathbb{C}}, \omega_{\mathrm{KKS}}) = 4\pi \cdot \mathrm{rk}(N_{\mathbb{C}}).
\]
\end{theorem}

Both values can however pretty easily be seen in terms of the Hamiltonian circle action induced by $H_\xi$. We recap this proof from \cite{B24} as we need parts of it later. The key input comes from Hwang and Suh \cite[Thm.\ 1.1]{HS13}.
\begin{theorem}[\cite{HS13}]\label{hs13}
    Let $(M,\omega)$ be a closed Fano symplectic manifold with a semifree Hamiltonian circle action. The Gromov width and the Hofer--Zehnder capacity are estimated as
    \begin{itemize}
        \item[(a)] $c_G(M,\omega) \leq \max(H) - \min (H) \leq c_{HZ}(M,\omega).$
        \item[(b)] Further if $H_{min}$ is a point, then
        $$
            c_G(M,\omega) = \mathrm{smin}(H)-\min(H),\ \ c_{HZ}(M,\omega)=\max(H)-\min(H).
        $$
    \end{itemize}
\end{theorem}
We will see that $H_\xi$ satisfies condition (b). Clearly $\xi\in\mathcal{O}_\xi$ is a critical point of $H_\xi$, indeed it is the minimum of $H_\xi$ and isolated.
\begin{lemma}[Lem.\ 4.6\ \cite{B24}]\label{isolated minimum}
The Hessian of $H_\xi$ at $p=\xi$ is positive definite, thus $p=\xi$ is an isolated local minimum. Indeed, $p=\xi$ is the global minimum.
\end{lemma}
Further we know the level sets, where critical points lie.
\begin{lemma}[Lem.\ 4.7\ \cite{B24}]\label{critical values}
    The Hamiltonian $H_\xi$ satisfies
    $$
    \max(H_\xi)-\min(H_\xi)=4\pi \cdot \mathrm{rk}(N_\C),\ \ \mathrm{smin}(H_\xi)-\min(H_\xi)=4\pi,
    $$
where $\mathrm{smin}(H_\xi)$ denotes the second lowest value of $H_\xi$ at a critical point.
\end{lemma}
These two lemmas together with Theorem \ref{hs13} prove Theorem \ref{capacities HSS}.

\subsection{Capacities of $U_rN$}\label{capacities UN}
We will now use the symplectic identification of Corollary \ref{symplectomorphism} to compute the Gromov width and the Hofer--Zehnder capacity of the $U$-neighborhoods.
\begin{theorem}
    Denote by \( \mathrm{sys} \) the length of the shortest closed geodesic of the symmetric \( R \)-space \( N \). Then
    $$
    c_G(U_1N, \dd\lambda) = c_{HZ}(U_1N, \dd\lambda) =
    \begin{cases} 
        \mathrm{sys}, & \text{if } \mathrm{rk}(N_\mathbb{C}) = 2 \cdot  \mathrm{rk}(N), \\ 
        2 \cdot \mathrm{sys}, & \text{if } \mathrm{rk}(N_\mathbb{C}) = \mathrm{rk}(N).
    \end{cases}
    $$
\end{theorem}

\begin{proof}
    The normalization $\omega_{KKS}(A)=4\pi=4\pi R^2$ implies $\mathrm{sys}= 2\pi=2\pi R$. So that the theorem reads
    $$
     c_G(U_rN,\dd\lambda)=c_{HZ}(U_rN,\dd\lambda)=4\pi,\ \ r=\frac{\mathrm{rk}(N_\C)}{\mathrm{rk}(N)}.
    $$
    We will prove this Theorem by showing independently that $4\pi\leq c_G$ and $c_{HZ}\leq 4\pi$. For the first inequality we will use the gradient flow of $H_\xi$ (from the previous section) to push $\Delta$ into the complement of balls of size approaching the Gromov width of $N_{\C}$ and obtain an embedding into $U_rN\cong N_\C\setminus\Delta$ with a relative Moser trick. The upper bound follows from a theorem by Lu \cite{Lu06} using a non-vanishing Gromov--Witten invariant of $N_\C$.\\
    \ \\
    \underline{\textit{Lower bound:}} The sublevel set $\lbrace H_\xi < 4\pi\rbrace\subset N_\C$ is symplectomorph to the open ball $B_2^{2n}$ of capacity $4\pi$ \cite[Prop.\ 2.11]{KT05} as the minimum of $H_\xi$ is isolated and the next critical value is $4\pi$ (cf.\ Lem. \ref{isolated minimum}, Lem. \ref{critical values}). Further the center of the ball is mapped to the minimum $\xi$. We may assume that $\xi\notin \Delta$, this implies that for large $T$ the gradient flow $\phi_T$ of $H_\xi$ pushes $\Delta$ into the complement of a slightly smaller ball $B_{2-\varepsilon}^{2n}\cong \lbrace H_\xi <\pi(2-\varepsilon)^2\rbrace $, i.e. 
    $$
   \exists T>0\ \mathrm{s.t.}\  \phi_T(\Delta)\subset N_\C\setminus B_{2-\varepsilon}^{2n}.
    $$
    Vice versa we obtain a symplectic embedding 
    $$
    (B^{2n}_{2-\varepsilon},\omega_0)\hookrightarrow (N_\C\setminus\phi_T(\Delta),\omega_{KKS})\cong (N_\C\setminus\Delta,\phi_T^*\omega_{KKS}).
    $$
    By Lemma \cite[Lem.\ 4.11]{B24} the gradient flow of $H_\xi$ is holomorphic which implies that $J$ is compatible with $\phi_T^*(\sigma\ominus R\sigma)$ for all $T$. Therefor $\Delta$ is a finite union of closed symplectic submanifolds for all symplectic structures in the family $\phi_T^*\omega_{KKS}$ and we can apply Moser's trick relative to $\Delta$ to identify 
    $$
    (N_\C\setminus\Delta,\phi_T^*\omega_{KKS})\cong (N_\C\setminus\Delta,\omega_{KKS}).
    $$
    Letting $\varepsilon\to 0$ we obtain the desired lower bound for the Gromov width.\\
    \ \\
    \underline{\textit{Upper bound:}} For the upper bound we use a theorem by Lu (cf.\ \cite[Thm.\ 1.10]{Lu06}) or rather a corollary of it (cf.\ \cite[Cor.\ A.1]{Bim23}), that says
    \begin{equation}\label{lu}
    \mathrm{GW}_{A,g,m+2}^{(M,\omega)}([pt.], [\Sigma],\ldots)\neq 0\Rightarrow c_{HZ}(M\setminus\Sigma,\omega)\leq \omega(A).
    \end{equation}
    Here, $(M,\omega)$ is closed symplectic, $A\in H_2(M,\Z)$, $g$ denotes the genus of the curves, $m+2$ is the number of marked points and $\Sigma\subset M$ is a (finite union of) submanifolds, representing a homology class $[\Sigma]\in H_*(M,\Q)$.\\
    In our case $(M,\omega)=(N_\C,\omega_{KKS})$, $A\in H_2(N_\C,\Z)$ is a generator\footnote{For an irreducible Hermitian symmetric space of compact type $H_2(N_\C,\Z)\cong \Z$. A proof of this fact is sketched in Appendix \ref{agenerator}.}, $g=0$, $\Sigma=\Delta$. As proved in \cite[Thm.\ 5.2]{CC16}, we indeed have a non-vanishing Gromov--Witten invariant
    $$
    \mathrm{GW}_{A,0,2}^{(N_\C,\omega_{KKS})}([pt.],\alpha)=1,
    $$
    for some class $\alpha\in H_*(N_\C,\Q)$. Since $\omega_{KKS}\vert_{N_\C\setminus\Delta}$ is exact it follows that the intersection number $A\cdot [\Delta]\neq 0$. Otherwise the energy of any curve $u:\cp^1\to N_\C,\ [u]=A$ must be zero $\int_{\cp^1} u^*\omega_{KKS}=0 $ by Stokes theorem, a contradiction as $\omega_{KKS}(A)=4\pi$ by normalization. But this implies that also 
    $$
    \mathrm{GW}_{A,0,2}^{(N_\C,\omega_{KKS})}([pt.],\alpha,[\Delta])=\mathrm{GW}_{A,0,2}^{(N_\C,\omega_{KKS})}([pt.],\alpha)\left (A\cdot[\Delta]\right)\neq 0.
    $$
    The desired upper bound follows from \eqref{lu}.
\end{proof}
\subsection{Capacities of $D_rN$}
This finishes the computation of Gromov width and Hofer--Zehnder capacity of the $U$-neighborhoods, but what about disc tangent bundles? Disc bundles are of particular interest, as they are sub level sets of the kinetic Hamiltonian $E(x,v)=\frac{1}{2}\vert v\vert_x^2$ and since $E$ generates the geodesic flow this reveals a deep connection to Riemannian geometry. If $N$ is simply connected we easily obtain the following corollary of Theorem \ref{capacities UN}.

\begin{theorem}
    If \( N \) is simply connected, then
    $
    c_{HZ}(D_1N, \dd\lambda) = \mathrm{sys}.
    $
\end{theorem}
\begin{proof}
    The upper bound immediately follows from Theorem \ref{capacities UN}, as $D_1N\hookrightarrow U_1N$. The lower bound is obtained modifying the Hamiltonian $H(x,v)=\mathrm{sys}\cdot \vert v\vert_x$ slightly near the zero section and near the boundary to obtain an admissible Hamiltonian (cf. \cite[Sec.\ 4]{Bim24.2}).
\end{proof}
We do not know the value of the Gromov width. Indeed, if $\mathrm{rk}(N)>1$ the ball we embedded into $N_\C$ as sublevel set of $H_\xi$, does not lie inside the disc bundle. \\
If $N$ is not simply connected, it is not even clear what the Hofer--Zehnder capacity should be. For real projective spaces $\rp^n$ (the only non-simply connected symmetric $R$-spaces of rank 1), the $U$-neighborhoods are disc bundles, thus 
$$
c_{HZ}(D_1\rp^n,\dd\lambda)=2\cdot\mathrm{sys}.
$$
On the other hand for real quadrics $Q_{p,q}(\R)$, it is not to hard to prove that 
$$
c_{HZ}(D_1Q_{p,q}(\R),\dd\lambda)=\sqrt{2}\cdot\mathrm{sys}.
$$
This is the length of the shortest contractible geodesic. For the upper bound observe that the universal covering is the double cover
$$
\mathrm{S}^p\times \mathrm{S}^q\to Q_{p,q}(\R)\cong \mathrm{S}^p\times \mathrm{S}^q/\Z_2,
$$
where $\Z_2$ acts via the diagonal antipodal map. Any shortest closed geodesic on $Q_{p,q}(\R)$ lifts to a diagonal geodesic segment joining antipodal points $(\pm a,\pm a)$. If we normalize the spheres to have radius one, this shortest closed geodesic has length $\sqrt{2}\pi$. 
Since, 
 $$
 D_1(\mathrm{S}^p\times \mathrm{S}^q)\subset D_1\mathrm{S}^p\times D_1 \mathrm{S}^q\cong Q^p\setminus Q^{p-1}\times Q^q\setminus Q^{q-1}
 $$
 we obtain 
 $$
 c_{HZ}(D_1Q_{p,q}(\R),\dd\lambda)\leq c_{HZ}(Q^p\setminus Q^{p-1}\times Q^q\setminus Q^{q-1},\frac{1}{2}(\omega_{KKS}\oplus\omega_{KKS}))\leq 2\pi=\sqrt{2}\cdot\mathrm{sys}.
 $$
The lower bound is obtained by approximating billiards on the product of two hemispheres in $\mathrm{S}^p\times \mathrm{S}^q$. The precise construction is completely analogous to the construction of billiards on a hemisphere carried out in \cite[Sec.\ 4]{Bim24.2} and is therefore omitted here. Note, that the shortest bounce orbit is the two bounce orbit, when the geodesic in one factor is constant. The length of such an orbit is $2\pi=\sqrt{2}\cdot\mathrm{sys}$.
\appendix
\section{The biholomorphism $TN\cong K^\C/H^\C$}\label{biholo}
The following biholomorphism is rather standard and appears in many places, though often implicitly. For example, as an equivariant diffeomorphism, it can be found in \cite{Tum23,AG90, Gea06}, and as a biholomorphism, it appears in \cite{BHH03, HI03}. For convenience, we include a sketch of the proof here.  

\begin{theorem}[\cite{AG90, HI03}] \label{complex identification tangent bundle}  
The space \( K^\C/H^\C \) is \( K \)-equivariantly biholomorphic to \( TN \), equipped with an adapted complex structure.  
\end{theorem}

\begin{proof}  
Denote \( \mathfrak{k} = \mathfrak{h} \oplus \mathfrak{l} \) the Cartan decomposition with respect to the symmetric pair \( (K, H) \). We identify \( TN \cong K \times_H \mathfrak{l} \), where \( H \) acts on the product via  
\[
h \cdot (k, X) = (kh^{-1}, \mathrm{Ad}_h(X)).
\]
Now, the desired \( K \)-equivariant biholomorphism is given by  
\[
\Phi: TN \cong K \times_H \mathfrak{l} \to K^\C/H^\C, \quad (k, X) \mapsto k\exp(iX)H^\C.
\]
A complex structure on \( TN \) is called \emph{adapted} if, for any geodesic \( \gamma: \mathbb{R} \to N \), the differential  
\[
\dd\gamma: T\mathbb{R} = \mathbb{C} \to TN, \quad (x, y) = x + iy \mapsto (\gamma(x), y\gamma'(x))
\]
is holomorphic. Any geodesic $\gamma:\R\to N$ (starting at the origin) is generated by an element $X\in \mathfrak{l}$, which indeed implies that $\Phi\circ \dd\gamma:\C\to N_\C$ is holomorphic, i.e.
$$\Phi(\dd\gamma(x,y))=\Phi(\exp(xX),yX)=\exp((x+iy)X)H^\C.$$
By counting dimensions, one may believe that the map is a local biholomorphism near the zero section. The global statement follows from Mostow's decomposition theorem \cite{Mos55}, which states that \( K^\C \) is equivariantly diffeomorphic to \( K \times \mathfrak{k} \cong TK \) for a compact Lie group \( K \).  
\end{proof}  
\section{Spheres as symmetric R-spaces}\label{spheres}
Spheres $\mathrm{S}^n=\lbrace x^2_1+\ldots+x_n^2=1\rbrace\subset \R^{n+1}$ are symmetric spaces, the associated symmetric pair is $(K,H)=(\mathrm{SO}(n+1), \mathrm{SO}(n))$. There is a larger Lie transformation group acting on $\mathrm{S}^n$: the group of conformal diffeomorphisms $G=\mathrm{SO}(n+1,1)$. To see how it acts, we identify the sphere as quadric in $\rp^{n+1}$:
$$
\mathrm{S}^n \to \lbrace [x] \in \rp^{n+1}\ : \ x^T S x=0\rbrace,\ \ (z,y)\mapsto(z-1, \sqrt{2} y, z+1),
$$
where
$$
S=\begin{pmatrix}
    0 & 0 & 1\\
    0 & 1_n & 0\\
    1 & 0 & 0
\end{pmatrix}.
$$
 The corresponding isometry group is $\mathrm{SO}(n+1,1)$ preserves the sphere and induces the conformal $\mathrm{SO}(n+1,1)$-action on $\mathrm{S}^n$ (cf.\ \cite[Sec.\ 2.1]{Sch08}. Note, that 
$$
C^T S C =D, \text{for}\ D=\begin{pmatrix}
    1 & 0 & 0\\
    0 & 1_n & 0 \\
    0 & 0 & -1 
\end{pmatrix}\ \text{and} \ C=\frac{1}{\sqrt{2}}\begin{pmatrix}
    1 & 0 & -1\\
    0 & 1_n & 0\\
    1 & 0 & 1
\end{pmatrix}.
$$
This means conjugation by $C$ intertwines the standard representation of $\mathrm{SO}(n+1,1)$ and the representation with respect to $S$. Next we want to determine the stabilzer of a point. Take $(z,y)=(-1,0)$, the image in $\rp^{n+1}$ is $[x]$ for $x=\begin{pmatrix}
    -2 & 0 & 0
\end{pmatrix}^T$. This means $(Cx)^T=\begin{pmatrix}
    -\sqrt{2} & 0 & -\sqrt{2}
\end{pmatrix}^T$, which is fixed by $P = \mathrm{SO}(n)\times \mathrm{SO}(1,1)$. We find that $G^\C=\mathrm{SO}(n+2,\C)$ and $P^\C=\mathrm{SO}(n,\C)\times \mathrm{SO}(2,\C)$. It follows that $G^\vee= \mathrm{SO}(n+2)$ and $K^\vee=\mathrm{SO}(n)\times \mathrm{SO}(2)$. Thus
$$
(\mathrm{S}^n)_\C\cong \mathrm{SO}(n+2,\C) / \mathrm{SO}(n,\C)\times \mathrm{SO}(2,\C)\cong \mathrm{SO}(n+2)/\mathrm{SO}(n)\times \mathrm{SO}(2)\cong Q_n(\C).
$$
This means we get two Cartan involutions of $\mathfrak{so}(n+2)$, the one associated with the symmetric pair $(\mathrm{SO}(n+2),\mathrm{SO}(n)\times\mathrm{SO}(2))$ and the one for the symmetric pair $(\mathrm{SO}(n+2), \mathrm{SO}(n+1))$ inducing the antiholomorphic involution on $Q_n(\C)$. Note that $\mathrm{S}^n$ sits as $\mathrm{SO}(n+1)$-orbit of an extrinsically symmetric element in the tangent space at any point of $\mathrm{S}^{n+1}$, the symmetric space associated to the symmetric pair $(\mathrm{SO}(n+2),\mathrm{SO}(n+1))$.

\vspace{1em}

The exact same construction can be carried out for the real quadrics of signature $(p,q)$:
$$
Q_{p,q}(\R):=\lbrace [x]\in\rp^{n+1}\ : x^TS_{p,q}x=0\rbrace,\ S_{p,q}=\frac{1}{\sqrt{2}}\begin{pmatrix}
    1 & 0 & 1\\
    0 & 1_{p,q} & 0\\
    1 & 0 & 1
\end{pmatrix},
$$
where $1_{p,q}$ is the diagonal matrix with p times 1 and q times -1. For a detailed discussion see \cite[Sec.\ 2.1]{Sch08}. We only note here that $Q_{p,q}(\R)\cong \mathrm{S}^p\times \mathrm{S}^q/\Z_2$ as symmetric space, where $\Z_2$ acts as the diagonal antipodal map.
\newpage
\section{List of symmetric R-spaces}\label{classification}
Indecomposable symmetric R-spaces are either irreducible Hermitian symmetric spaces of compact type or real forms of irreducible Hermitian symmetric spaces of compact type. Both are classified and for the readers convenience we include the lists here.
\subsection{Irreducible Hermitian symmetric spaces of compact type \cite{HG01}} For Hermitian symmetric spaces $N_\C=N\times N$, hence $\mathrm{rk}(N_\C)=2\cdot\mathrm{rk}(N)$. Further they are all simply connected $\pi_1(N)=0$.
\begin{enumerate}
    \item Complex Grassmannians: $$ N=\mathrm{Gr}_\C(p,p+q)=\mathrm{SU}(p+q)/\mathrm{SU}(p)\times \mathrm{SU}(q)$$
    \item Space of orthogonal complex structures on $\R^{2n}$: 
    $$
    N=\mathrm{SO}(2n)/\mathrm{U}(n)
    $$
    \item Space of complex structures on $\mathbb{H}^{n}$ compatible with the inner product:
    $$
    N=\mathrm{Sp}(n)/\mathrm{U}(n)
    $$
    \item Complex quadrics:
    $$
    N=Q_n(\C)=\mathrm{SO}(n+2)/\mathrm{SO}(n)\times\mathrm{SO}(2)
    $$
    \item Complexification of Cayley projective plane $\mathbb{O}P^2$:
    $$
    N=(\C\otimes \mathbb{O})P^2=E_6/\mathrm{SO}(10)\times\mathrm{SO}(2)
    $$
    \item Space of symmetric submanifolds of Rosenfeld projective plane $(\mathbb{H}\otimes \mathbb{O})P^2$ isomorphic to $(\C\otimes \mathbb{O})P^2$:
    $$
    N=E_7/E_6\times \mathrm{SO}(2)
    $$
    \end{enumerate}
\newpage
\subsection{Indecomposable symmetric R-spaces of non-Hermitian type \cite{Tak84}}\label{table}
\begin{small}
\begin{center}
\begin{tabular}{|c| c c c c|} 
 \hline
  & N & $N_\C$ & $\pi_1(N)$ & $\mathrm{rk}(N_\C)/\mathrm{rk}(N)$ \\ [0.5ex] 
 \hline\hline
 (1) & $\mathrm{SO(p+q)/\mathrm{S}(\mathrm{O}(p)\times \mathrm{O}(q)})$ & $\mathrm{SU(p+q)/\mathrm{S}(\mathrm{U}(p)\times \mathrm{U}(q)})$  &  &   \\
 \hline
 (a) & $\rp^1$ & $\cp^1$ & $\Z$ & 1  \\
 p=q=1 & & & &\\
 \hline
 (b) & $\mathrm{Gr}_\R(p,p+q)$ & $\mathrm{Gr}_\C(p,p+q)$ & $\Z_2$ & 1 \\
 $q\geq p\geq 2$ & & & &\\
 \hline
 (2) & $\mathrm{Gr}_{\mathbb{H}}(p,p+q)$& $\mathrm{Gr}_{\mathbb{C}}(2p,2p+2q)$& 0 & 2 \\
 $q\geq p\geq 1$ & & & &\\
 \hline
 (3) & $\mathrm{U}(n)$ & $\mathrm{Gr}_\C(n,2n)$ & $\Z$ & 1 \\
 $n\geq 2$ & & & & \\
 \hline
 (4) & $\mathrm{SO}(n)$ & $\mathrm{SO}(2n)/\mathrm{U}(n)$ & $\Z_2$ & 1 \\  
 $n\geq 5$ & & & &\\
 \hline
 (5) & $\mathrm{U}(2n)/\mathrm{Sp}(n)$ & $\mathrm{SO}(4n)/\mathrm{U}(2n)$ & $\Z$ & 1 \\
 $n\geq 3$ & & & & \\
 \hline
 (6) & $\mathrm{Sp}(n)$ & $\mathrm{Sp}(2n)/\mathrm{U}(2n)$ & 0 & 2\\
 $n\geq 2$ & & & & \\
 \hline
 (7) & $\mathrm{U}(n)/\mathrm{O}(n)$ & $\mathrm{Sp}(n)/\mathrm{U}(n)$ & $\Z$ & 1 \\
 $n\geq 3$ & & & & \\
 \hline
 (8) & $Q_{p,q}(\R)$ & $Q_{p+q}(\C)$ & & \\
 \hline
 (a) & $\mathrm{S}^n$ &  & 0 & 2 \\
 $p=0,q\geq 2$ & & & & \\
 \hline
 (b)  & $\mathrm{S}^1\times \mathrm{S}^q/\Z_2$ & & $\Z$ & 1\\
 $p=1,q\geq 2$ & & & &\\
 \hline
 (c)  & $\mathrm{S}^p\times \mathrm{S}^q/\Z_2$ & & $\Z_2$ & 1\\
 $p\geq q\geq 2$ & & & &\\
 \hline
 (9) & $\mathrm{Gr}_{\mathbb{H}}(2,4)/\Z_2$ & $E_6/U(1)\times\mathrm{Spin}(10)$ & $\Z_2$ & 1\\
 \hline
 (10) & $\mathbb{O}\mathrm{P}^2$ & $E_6/U(1)\times\mathrm{Spin}(10)$ & 0 & 2\\
 \hline
 (11) & $\mathrm{SU}(8)/\mathrm{Sp}(4)\cdot \Z_2$ & $E_7/U(1)\times E_6 $ & $\Z_2$ & 1\\
 \hline
 (12) & $U(1)\times E_6/ F_4$ & $E_7/U(1)\times E_6 $& $\Z$ & 1\\
 \hline
\end{tabular}
\end{center}
\end{small}
\section{Homology of Hermitian symmetric spaces}\label{agenerator}
The following lemma is considered general knowledge and often stated in the literature (cf.\ \cite[Ch.\ 5 sec.\ 16]{BH58}), but we could not find a proof, so we present one here.
\begin{lemma}\label{generator}
    Let $N\cong K/H$ be an irreducible Hermitian symmetric space of compact type, then
    $$
    H_2(N,\Z)\cong \Z.
    $$
\end{lemma}
\begin{proof}
    We can construct a (Morse) chain complex with non vanishing chain groups only in even degrees as follows. The Hamiltonian $H_\xi: N\to \R$ from Section \ref{Hamcircle} is Morse-Bott with only even indices, as it generates a Hamiltonian $\mathrm{S}^1$-action \cite[Lem.\ 5.5.8]{DS17}. The critical submanifolds are finite unions of totally geodesic complex submanifolds, hence Hermitian symmetric spaces themselves. In particular they admit Morse--Bott Hamiltonians just like $H_\xi$, decomposing $N$ recursively only using even indices. The resulting complex has non trivial chain groups only in even degrees, thus all differentials vanish. In particular the chain groups coincide with the homology groups and therefor no torsion occurs. We will now compute the de Rham cohomology group $H^2_{dR}(N,\R)\cong H^2(N,\R)\cong H_2(N,\R)\cong \R$. As there is no torsion, universal coefficient theorem implies that $H_2(N,\Z)\cong \Z$.\\
    
    Every de Rham cohomology class can be represented by a $K$-invariant form, averaging if necessary.
    We need to show that there is up to scalar multiple only one invariant closed 2-form. Let $\nu\in \Omega_2(N)$ be any $K$-invariant 2-form. Note that $K$-invariance implies, that $\nu$ is fully determined by its value at the origin $o\in N$. Define an $H$-invariant symmetric operator $A:\mathfrak{l}\to\mathfrak{l}$ implicitly via
    $
    \nu_o(A\cdot,\cdot)=(\omega_{KKS})_o(\cdot,\cdot),
    $
    where $\mathfrak{k}=\mathfrak{h}\oplus\mathfrak{l}$ is the Cartan decomposition associated to the symmetric pair $(K,H)$. Then, $A$ must be a multiple of the identity, because the representation of $H$ on $\mathfrak{l}$ is irreducible and any eigenspace defines an invariant subspace. 
\end{proof}

\thispagestyle{empty}


\bibliography{literatur}
\bibliographystyle{alpha}
\end{document}